\documentclass[a4paper, 11pt]{amsart}

\usepackage[T1]{fontenc}
\usepackage[utf8]{inputenc}
\usepackage{geometry}
\usepackage{amssymb}

\makeatletter
\@namedef{subjclassname@2020}{%
  \textup{2020} Mathematics Subject Classification}
\makeatother

\allowdisplaybreaks[2]

\usepackage{tikz}
\usetikzlibrary{matrix,arrows}

\usepackage[all]{xy}

\usepackage{marvosym}

\newtheorem{theorem}{Theorem}[section]
\newtheorem{proposition}[theorem]{Proposition}
\newtheorem{lemma}[theorem]{Lemma}
\newtheorem{corollary}[theorem]{Corollary}
\theoremstyle{definition}

\newtheorem{example}[theorem]{Example}

\numberwithin{equation}{section}

\newcommand\FX{F\langle X\rangle}

\newcommand\NN{\mathbb{N}}

\newcommand\oA{\overline{A}}

\newcommand\ve{\varepsilon}
\begin{document}

\title[Jordan homomorphisms and T-ideals]{Jordan homomorphisms and T-ideals}

\author{Matej Brešar} 
\author{Efim Zelmanov}
\address{Faculty of Mathematics and Physics, University of Ljubljana \&
Faculty of Natural Sciences and Mathematics, University of Maribor \& IMFM, Ljubljana, Slovenia}
\email{matej.bresar@fmf.uni-lj.si}
\address{SICM, Southern University of Science and Technology,
Shenzhen, China}
\email{efim.zelmanov@gmail.com}

\thanks{The first author was partially supported by the ARIS Grants P1-0288 and J1-60025.  
The second author was partially supported by the NSFC Grant 12350710787 and the Guangdong Program 2023JC10X085.
}

\begin{abstract}
    Let $A$ and $B$ be associative algebras over a field $F$ with {\rm char}$(F)\ne 2$.
    Our first main result states that if $A$ is unital and equal to its commutator ideal,
    then every Jordan epimorphism
    $\varphi:A\to B$ is the 
     sum of a homomorphism and an antihomomorphism. Our second main result concerns (not necessarily surjective)  Jordan homomorphisms  from
     $H(A,*)$ to $B$, where $*$ is an involution
     on $A$ and $H(A,*)=\{a\in A\,|\, a^*=a\}$.
We show that there exists
a ${\rm T}$-ideal $G$ having the following two properties:
(1) the Jordan homomorphism
$\varphi:H(G(A),*)\to B$ can be extended to an (associative) homomorphism, subject to the condition  that the subalgebra generated by $\varphi(H(A,*))$
    has trivial annihilator, and
    (2)
 every element of the ${\rm T}$-ideal of identities of the algebra of $2\times 2$ matrices is nilpotent modulo $G$. A similar statement is true for Jordan homomorphisms from $A$ to $B$. A counter-example shows that the assumption on trivial annihilator cannot be removed.
\end{abstract}


\maketitle

\section{Introduction}

Let $A$ and $B$ be associative algebras over a field $F$ with char$(F)\ne 2$. Recall that a linear subspace $J$ of $A$  that is closed under the Jordan product $x\circ y=xy+yx$ is called a special Jordan algebra. 
A linear mapping $\varphi:J\to B$ is called a {\bf Jordan homomorphism}
if
\begin{equation}
\label{j1}\varphi(x\circ y)=\varphi(x)\circ \varphi(y)
\end{equation}
for all $x,y\in J$. We remark that
 while   $J$ can be any special Jordan algebra in this definition, the most studied and arguably most important cases are  when $J=A$  and $J=H(A,*)$,  the latter being the set of symmetric elements in  the algebra $A$ with an involution $*:A\to A$. We will also confine ourselves to these two cases.

 The restrictions of (associative) homomorphisms and antihomomorphisms  $A\to B$ to $J$ are obviously Jordan homomorphisms. Finding conditions under which a Jordan homomorphism $J\to B$ arises from
 homomorphisms and antihomomorphisms is a classical problem that originated in the 1940-50s in the works of G.\ Ancochea \cite{An1, An2}, I.\ Kaplansky \cite{Kap}, L.-K.\ Hua \cite{Hua}, N.\ Jacobson and C.\ E.\ Rickart \cite{JR, JR2}, 
 R.\ V.\ Kadison \cite{K}, and
  I.\ N.\ Herstein \cite{Her}.
This problem is natural from the point of view of the Jordan theory, and is independently important since Jordan homomorphisms naturally appear in a variety of mathematical topics. 

For more information about the history and background of the study of Jordan homomorphisms we
 refer the reader to \cite{B, FIbook, Mc, Mol}. In the next few lines we 
 will  briefly describe the results from two papers (\cite{Mc} and \cite{B}) that are relevant for our work. 

In \cite{Z}, the second author constructed the so called ``tetrad-eating'' T-ideal $T$ of the free Jordan algebra. In \cite{Mc}, 
K. McCrimmon showed that if $J$ is a unital  special Jordan algebra such that
$J=T(J)$, then all Jordan homomorphisms $J\to B$ are of the usual form.

In \cite{B}, the first author proved that if a Jordan homomorphism $\varphi:A\to B$
has the property
that  the ideal  generated by all $\varphi(xy)-\varphi(x)\varphi(y)$, $x,y\in A$, has trivial intersection with the ideal generated by 
 all $\varphi(xy)-\varphi(y)\varphi(x)$, $x,y\in A$ (such a Jordan homomorphism is called {\bf splittable}), then the restriction of $\varphi$
 to the {\bf commutator ideal} $K(A)$ (= the ideal generated by all commutators $[x,y]=xy-yx$, $x,y\in A)$
 has the usual form. The restriction to $K(A)$ is justified by examples.

We will extend the results from \cite{Mc} and \cite{B} in two directions. To describe our results, we first introduce some terminology.

 We will refer to Jordan homomorphisms having the usual form as ``standard''. To define this precisely, we separate the case where $J=A$ and the case where $J=H(A,*)$.
 \begin{enumerate}
     \item[(a)] A Jordan homomorphism $\varphi:A\to B$  is {\bf standard} if there exist a homomorphism $\varphi_1:A\to B$ and an antihomomorphism $\varphi_2:A\to B$ such that \begin{equation*} \label{ha0}   \varphi=\varphi_1+\varphi_2\quad\mbox{and}\quad  \varphi_1(A)\varphi_2(A)= \varphi_2(A)\varphi_1(A)=(0). \end{equation*}
     \item[(b)] 
  A Jordan homomorphism $\varphi:H(A,*)\to B$  is {\bf standard}   
if it can be extended to a homomorphism from $A$ to $B$.
 \end{enumerate}
 (Note that if $\Phi:A\to B$ is an  antihomomorphism, then $x\mapsto \Phi(x^*)$ is a homomorphism that coincides with $\Phi$ on $H(A,*)$. That is why there are no antihomomorphisms in (b).)

In Section \ref{sec2} we show
  that if $A$ is any unital algebra such that $A=K(A)$, then every Jordan homomorphism from 
$A$ onto an arbitrary  algebra $B$ is standard (Theorem \ref{mt1}). Thus, unlike in \cite{B}, no condition on $\varphi$ is required. 
Also,  condition $A=T(A)$ from \cite{Mc} obviously
implies $A=K(A)$, but not vice-versa. Moreover,
 the final result of Section \ref{sec2} (Corollary \ref{cfin})
gives some evidence that  condition $A=K(A)$ is natural and in some sense optimal. We also mention that our method of proof in Section \ref{sec2} is different from those in  \cite{B} and \cite{Mc}.

 For a subset $S$ of $A$, we write
 $\langle S\rangle$ for the
(associative) subalgebra of $A$
generated by $S$. 
In Section \ref{sec3} we prove the existence of a T-ideal $G$
such that every element of the ${\rm T}$-ideal of identities of the algebra of $2\times 2$ matrices is nilpotent modulo $G$,
and every Jordan homomorphism $A\to B$ (resp.\ $H(A,*)\to B$)
is standard on $G(A)$ (resp.\
$H(G(A),*)$), provided that the subalgebra $\langle \varphi(A)\rangle$ (resp.\ $\langle \varphi(H(A,*))\rangle$)  has trivial annihilator (Theorem \ref{te}). The latter restriction on $\varphi$ is very mild and easy to handle  compared to the restriction from
\cite{B} that $\varphi$ is splittable, and, moreover, we show that it is necessary (Proposition \ref{pez}). 
Of course, the  
main difference and advantage when compared with the result from \cite{Mc} is that $G(A)$ 
(resp.\ $H(G(A),*)$) is not assumed to be equal to $A$ (resp.\ $H(A,*)$). 
As in \cite{Mc}, the main idea of the proof is to use the result from \cite{Z}. However, the realization of this  idea is somewhat different.

\section{Jordan homomorphisms on algebras equal to their commutator ideals}\label{sec2}

In this section,
we consider a Jordan homomorphism $\varphi:A\to B$ where 
$A$ and $B$ are associative algebras over a field $F$.

Recall our convention that char$(F)\ne 2$. We assume this without further mention.  
In some of the 
results---specifically, in Lemmas \ref{le2}, \ref{le3}, \ref{lef1}, \ref{ts}, and \ref{lbid}---this assumption actually is not needed. Here it should be mentioned that
a Jordan homomorphism between algebras 
over a field with characteristic $2$ is defined as a linear map $\varphi$ satisfying two conditions,
\begin{equation}
\label{jhtri}\varphi(x^2)=\varphi(x)^2\quad\mbox{and}\quad\varphi(xyx)=\varphi(x)\varphi(y)\varphi(x).
\end{equation}
In our setting where char$(F)\ne 2$  the first condition is readily equivalent to \eqref{j1} and implies the second condition since $xyx= \frac{1}{2}(x\circ (y\circ x) - x^2\circ y)$.

We will now record a few 
elementary observations regarding any Jordan homomorphism $\varphi:A\to B$. First, 
from \eqref{jhtri} we infer that $$\varphi(xyz+zyx)=\varphi(x)\varphi(y)\varphi(z) +\varphi(z)\varphi(y)\varphi(x)\quad\mbox{and}\quad \varphi(y^i)=\varphi(y)^i$$  for all  $i\ge 1$. 
Hence, $$\varphi(y^iwy^j + y^jwy^i)=\varphi(y)^i\varphi(w)\varphi(y)^j + \varphi(y)^j\varphi(w)\varphi(y)^i$$
and therefore
\begin{equation}\label{ene}
    \varphi\left(\sum_{i=0}^{m} y^i w y^{m-i}\right)=\sum_{i=0}^{m} \varphi(y)^i \varphi(w) \varphi(y)^{m-i}
\end{equation}
for every $m\in\NN$. 

Another standard equation,
\begin{equation}\label{deve}
    \varphi([x,[u,v]])=
    [\varphi(x),[\varphi(u),\varphi(v)]], 
\end{equation} is an immediate  consequence of 
 \begin{equation}\label{uava}
     [x,[u,v]] = (x\circ u)\circ v - (x\circ v)\circ u.
 \end{equation}

 Take $u,v\in A$ and write
     \begin{equation}\label{abuv}a=\varphi(uv)-\varphi(u)\varphi(v)\quad\mbox{and}\quad b=\varphi(uv)-\varphi(v)\varphi(u).
     \end{equation}
        The following two equations are  less obvious but also well known:
\begin{equation}\label{abba}
    ab=ba=0
\end{equation}
and
\begin{equation}\label{abba2}
    a\varphi(x)b+b\varphi(x)a=0
\end{equation}
for every $x\in A$. The first one was discovered in \cite{JR} and the second one in \cite{Sm}.

\begin{lemma}\label{le2}
   Let $\varphi:A\to B$ be a Jordan homomorphism.
    Let $u,v\in A$ and write $y=[u,v]$,
    $\tilde{y}= [\varphi(u),\varphi(v)]$. Then
    $$\varphi([x,y^n])= [\varphi(x),\varphi(y)^{n-1}\tilde{y}]$$ for all $x\in A$ and all $n\in\NN\setminus{\{2\}}$.
\end{lemma}

\begin{proof} 
By \eqref{ene}, we have
\begin{equation}\label{beg}
    \varphi([x,y^n])=\varphi\left(
\sum_{i=0}^{n-1} y^i[x,y]y^{n-i-1}
\right)=\sum_{i=0}^{n-1} \varphi(y)^i\varphi([x,y])\varphi(y)^{n-i-1}.
\end{equation}
From \eqref{deve} we see that
$\varphi([x,y])= [\varphi(x),\tilde{y}]$, i.e.,
 the assertion of the lemma is true for $n=1$. Hence, \eqref{beg} becomes
\begin{equation}\label{fr}
\varphi([x,y^n])=
\sum_{i=0}^{n-1} \varphi(y)^i[\varphi(x),\tilde{y}]\varphi(y)^{n-i-1}.
\end{equation}

Using the notation introduced in \eqref{abuv}, we have
\begin{align*}
    b-a&= [\varphi(u),\varphi(v)] =\tilde{y},\\b+a &= 2\varphi(uv) - \varphi(u\circ v)=\varphi(y).
\end{align*}
In light of  \eqref{abba}, we see that 
 \begin{align*}
     & \varphi(y)^k = (b+a)^k=b^k + a^k,\\
    &\varphi(y)^k\tilde{y}=\tilde{y}\varphi(y)^k= (b^k+a^k)(b-a)=(b-a)(b^{k}+a^k)= b^{k+1}-a^{k+1} 
 \end{align*}
 for every $k\ge 1$.

Assume now that $n\ge 3$. We 
can rewrite \eqref{fr} as
\begin{align*}
\varphi([x,y^n])=& [\varphi(x),(b-a)](b^{n-1}+a^{n-1}) \\&+ \sum_{i=1}^{n-2} (b^i+ a^i)[\varphi(x),(b-a)] (b^{n-i-1}+ a^{n-i-1})\\
&+(b^{n-1}+a^{n-1})[\varphi(x),(b-a)]\\
=&\varphi(x)(b^n-a^n) - (b-a)\varphi(x)(b^{n-1}+ a^{n-1})\\&+\sum_{i=1}^{n-2} (b^i + a^i)\varphi(x)(b^{n-i} 
- a^{n-i}) -
\sum_{i=1}^{n-2} (b^{i+1} - a^{i+1})\varphi(x)(b^{n-i-1}+ a^{n-i-1})\\ &+ (b^{n-1}+a^{n-1})\varphi(x)(b-a) - (b^n-a^n)\varphi(x).
\end{align*}
Observe that \eqref{abba} and \eqref{abba2} imply 
$$b^j\varphi(x)a^{n-j}= a^j\varphi(x)b^{n-j}=0,\quad 1\le j\le n-1.$$  Using this together with $b^n-a^n= \varphi(y)^{n-1}\tilde{y}$, it follows that
\begin{align*}
    &\varphi([x,y^n])- [\varphi(x),\varphi(y)^{n-1}\tilde{y}]
    \\= &a\varphi(x)a^{n-1} -\sum_{i=1}^{n-2} a^i\varphi(x)a^{n-i} + \sum_{i=1}^{n-2} a^{i+1}\varphi(x)a^{n-i-1} -a^{n-1}\varphi(x)a
    \\& - b\varphi(x)b^{n-1} +\sum_{i=1}^{n-2}b^i\varphi(x)b^{n-i} - \sum_{i=1}^{n-2} b^{i+1}\varphi(x)b^{n-i-1} + b^{n-1}\varphi(x)b\\
    =&0,
\end{align*}
as desired.
\end{proof}

We continue by introducing an auxiliary  notion that will naturally appear in what follows. We will say that a Jordan homomorphism 
$\varphi:A\to B$ is {\bf nearly-standard} if there exists
mappings $\varphi_1,\varphi_2:A\to B$ such that
\begin{equation}\label{ha}\varphi=\varphi_1+\varphi_2\quad\mbox{and}\quad\varphi(xy)=\varphi(x)\varphi_1(y) + \varphi_2(y)\varphi(x)\end{equation}
for all $x,y\in A$ (we are not assuming that $\varphi_1$ is  a homomorphism  and $\varphi_2$ 
 an antihomomorphism). 
 
 Observe that a standard Jordan homomorphism is nearly-standard. 
If $A$ is commutative, then
every Jordan homomorphism on $A$ is automatically nearly-standard (with $\varphi_1=\varphi_2=\frac{\varphi}{2}$), but 
not always standard, as we will see in Example \ref{snx}.
An example of a noncommutative algebra $A$ and a non-standard but nearly-standard Jordan homomorphism $\varphi:A\to B$  will be given in 
Example \ref{examb}.

\begin{lemma}
 \label{ts0}
     If an algebra $A$ is equal to the linear span of  its commutators and their 
    $n$th powers for $n\ge 3$, then
     every Jordan homomorphism $\varphi:A\to B$, where $B$ is any $F$-algebra, is nearly-standard.
\end{lemma}

\begin{proof} From
    Lemma \ref{le2} we see that our assumption on $A$ implies 
  that there is a function
   $\psi:A\to B$ such that
   $$\varphi([x,y])=[\varphi(x),\psi(y)]$$ for all $x,y\in A$. Since $xy=\frac{1}{2}x\circ y + \frac{1}{2}[x,y]$, $$\varphi_1 =\frac{1}{2}(\varphi + \psi)\quad\mbox{and}\quad \varphi_2 =\frac{1}{2}(\varphi - \psi)$$ satisfy \eqref{ha}.
   \end{proof}

Our next goal is to replace the assumption  on $A$ in Lemma \ref{ts0} by a more 
transparent assumption that $A$ is unital and  equal to its commutator ideal $K(A)$ (Proposition \ref{pp}). To this end, we need two  lemmas.

\begin{lemma}\label{le3}
    Let  $n\in\NN$ be such that {\rm char}$(F)$ 
    does not divide $n$.
    If a linear subspace $V$ of an algebra $A$ contains all commutators and their $n$th powers, then $V$ contains the ideal of 
    $A$ generated by elements $$[u,\sum_{\sigma\in S_{n-1}} [u,v_{\sigma(1)}]\cdots [u,v_{\sigma(n-1)}]],$$ $u,v_1,\dots,v_{n-1}\in A$. 
\end{lemma}

\begin{proof}
   By assumption, $[u,v]^n \in V$ for all $u,v\in A$.   Linearizing in $v$, we obtain
\begin{equation}\label{pom}
\sum_{\sigma\in S_n} [u,v_{\sigma(1)}]\cdots [u,v_{\sigma(n)}] \in V    
\end{equation}
for all $u,v_1,\dots,v_n\in A$. Write $w=v_n$
and  $$\overline{v}_{n-1}=\sum_{\sigma\in S_{n-1}} [u,v_{\sigma(1)}]\cdots [u,v_{\sigma(n-1)}].$$
Since $V$ contains the commutators, \eqref{pom}
implies $n[u,w]\overline{v}_{n-1}\in V$, and therefore
$$[u,w]\overline{v}_{n-1}\in V$$
   by the characteristic assumption.
   Consequently, for any $u,w, v_i\in A$, 
   $$w[u,\overline{v}_{n-1}] = [u,w\overline{v}_{n-1}] - [u,w]\overline{v}_{n-1} \in V,$$
   and hence also
     $$w[u,\overline{v}_{n-1}]z = [w[u,\overline{v}_{n-1}],z] + (zw)[u,\overline{v}_{n-1}]\in V $$ for all 
     $u,w,z,v_i\in A$.   This proves the lemma.
\end{proof}

Let $X=\{x_1,x_2,\dots\}$ and let $F\langle X\rangle$ be the free algebra on the set $X$. For an element $f\in F\langle X\rangle$, let $(f)^{\rm T}$ denote the T-ideal generated by  $f$,
i.e.,  $(f)^{\rm T} = \FX f(\FX)\FX$.

We call a T-ideal $ T$ of $\FX$ a {\bf nonmatrix commutator T-ideal} if $T\subseteq ([x_1,x_2])^{\rm T}$ and $T(M_2(F))\ne (0)$.  We also assume that
for any of its  elements, $T$ also contains all its partial linearizations. This automatically holds if the field $F$ is infinite.

\begin{lemma} \label{lef1} Let $T_1,T_2$ be
two nonmatrix commutators {\rm T}-ideals. If $A$ is any unital algebra, then
 $A=T_1(A)$ if and only if $A=T_2(A)$. 
\end{lemma}

\begin{proof}
    Without loss of generality, we may assume that
    the field $F$ is algebraically closed and 
    card$(F) > \dim_F(A)$.

    Let $A=T_1(A)$. Consider the factor algebra
    $\oA=A/T_2(A)$. Our goal is to show that   $\oA=(0)$. We remark that
    $\oA = T_1(\oA)$.

    Let $\oA'$ be a primitive homomorphic image of $\oA$. By Kaplansky's theorem (see, e.g., \cite[Corollary 23.31]{Row}), $\oA' \cong M_k(D)$ where $D$ is a finite-dimensional division algebra over $Z(D)$. Since    card$(F) > \dim_F(D)$ and the field $F$ is algebraically closed, $D=F$ by the result by Amitsur (see, e.g., \cite[Corollary 15A.8]{Row}). The algebra $M_k(F)$ satisfies all identities from $T_2$. Hence, $k=1$ and 
    $\oA'\cong F$.

    Let $J(\oA)$ be the Jacobson ideal of the algebra $\oA$. We showed that the factor algebra $\oA/J(\oA)$ is commutative. Now,
    $\oA = T_1(\oA)$ and $T_1\subseteq ([x_1,x_2])^{\rm T}$ imply $\oA = J(\oA)$. Since $1\in\oA$ we conclude that $\oA=(0)$, as desired.
\end{proof}

\begin{proposition}
 \label{pp}
    If $A$ is a unital algebra such that $A=K(A)$, then
     every Jordan homomorphism $\varphi:A\to B$, where $B$ is any $F$-algebra, is nearly-standard.
\end{proposition}

\begin{proof}
    Note that   $$T_1= ([x_4, \sum_{\sigma\in S_{3}} [x_4,x_{\sigma(1)}]\cdots [x_4,x_{\sigma(3)}]])^{\rm T}$$  is a nonmatrix commutator T-ideal. 
     Obviously, so is $T_2=([x_1,x_2])^{\rm T}$.  Lemma \ref{lef1} thus shows that $A=K(A) =T_2(A)$ implies $A=T_1(A)$. The $n=4$ case of Lemma \ref{le3} further shows that in this case $A$ is equal to the linear span of commutators and their $4$th powers. We can now apply Lemma \ref{ts0}.
\end{proof}

Our next goal is to provide conditions under which nearly-standard Jordan homomorphisms are standard. We need the following definition.

Let $B$ be an algebra.
A  mapping $\Delta:B\times B\to B$ is called a {\bf biderivation} if
it is a derivation in each argument.
A biderivation $\Delta$ is said to be {\bf inner} if there exists an element $\lambda\in Z(B)$ such that 
$$\Delta(r,s)=\lambda [r,s]$$ for all
$r,s\in B$.

\begin{lemma}
 \label{ts}
  Let $B$ be a unital  algebra such that every biderivation of $B$ is inner and for every $\alpha\in Z(B)$, $\alpha[[B,B],B]=(0)$ implies $\alpha=0$. Then every nearly-standard Jordan isomorphism $\varphi:A\to B$, where $A$ is any algebra, is standard. Moreover, 
  there exists an  idempotent $\ve \in Z(B)$ such that $x\mapsto \ve \varphi(x)$ is a homomorphism and $x\mapsto (1-\ve) \varphi(x)$ is an antihomomorphism.
\end{lemma}

\begin{proof}
Let $\varphi_1$ and $\varphi_2$ be mappings satisfying $$\varphi=\varphi_1+\varphi_2\quad\mbox{and}\quad \varphi(xy)=\varphi(x)\varphi_1(y)+ \varphi_2(y)\varphi(x).$$ These equations imply
$$\varphi(xy) - \varphi(x)\varphi(y) = [\varphi_2(y),\varphi(x)].$$
Replacing the roles of $x$ and $y$, we obtain
$$\varphi(yx) - \varphi(y)\varphi(x) = [\varphi_2(x),\varphi(y)].$$
The sum of the left-hand sides is $0$ 
since $\varphi$ is a Jordan homomorphism. Therefore,
$$[\varphi(x),\varphi_2(y)]=[\varphi_2(x),\varphi(y)].$$
Here, $x$ and $y$ are arbitrary elements in $A$. Writing $r=\varphi(x)$ and $s=\varphi(y)$, we obtain
$$[r,\varphi_2(\varphi^{-1}(s))] = [\varphi_2(\varphi^{-1}(r)),s] $$
    for all $r,s\in B$.   The mapping
$(r,s)\mapsto [r,\varphi_2(\varphi^{-1}(s))]$ is therefore a biderivation of $B$. By assumption, there exists a $\lambda\in Z(B)$ such that  
$$[r,\varphi_2(\varphi^{-1}(s))] =\lambda[r,s]$$
for all $r,s\in B$. Hence, $$\varphi_2(\varphi^{-1}(s))  - \lambda s\in Z(B)$$ for every $s\in B$. Writing $s=\varphi(y)$ and $$\mu(y)=\varphi_2(y) - \lambda\varphi(y),$$ we see that  
$\mu(y)\in Z(B)$
for every $y\in A$. Since $\varphi=\varphi_1+\varphi_2$, we have $$\varphi_1(y)=(1-\lambda)\varphi(y)-\mu(y).$$  Set  $\ve =1-\lambda$ and note that $\varphi(xy)=\varphi(x)\varphi_1(y)+ \varphi_2(y)\varphi(x)$ can now be written as
\begin{equation}\label{s3}
    \varphi(xy)=\ve \varphi(x)\varphi(y) + (1-\ve)\varphi(y)\varphi(x)
\end{equation}
for all $x,y\in A$. Consequently,
\begin{align*}  \varphi(xyz)=&\varphi(x(yz)) \\=& \ve \varphi(x)\varphi(yz)  
    + (1-\ve)\varphi(yz)\varphi(x)\\
    =&\ve^2 \varphi(x)\varphi(y)\varphi(z)+
    \ve(1-\ve)\varphi(x)\varphi(z)\varphi(y)\\&+\ve(1-\ve)\varphi(y)\varphi(z)\varphi(x) + (1-\ve)^2 \varphi(z)\varphi(y)\varphi(x).
\end{align*}
On the other hand,
\begin{align*}
    \varphi(xyz)=&\varphi((xy)z) \\=& \ve \varphi(xy)\varphi(z)  
    + (1-\ve)\varphi(z)\varphi(xy)\\
    =&\ve^2 \varphi(x)\varphi(y)\varphi(z)+
    \ve(1-\ve)\varphi(y)\varphi(x)\varphi(z)\\&+\ve(1-\ve)\varphi(z)\varphi(x)\varphi(y) + (1-\ve)^2 \varphi(z)\varphi(y)\varphi(x).
\end{align*}
Comparing both expressions, we obtain 
$$\ve(1-\ve)[[\varphi(x),\varphi(z)],\varphi(y)] =0$$for all $x,y,z\in A$. 
By our assumption,  this implies
 $\ve(1-\ve)=0$. Thus,
$\ve$ is a central idempotent.
Now, multiplying \eqref{s3} by $\ve$ we see that $x\mapsto\ve\varphi(x)$ is a homomorphism, and multiplying \eqref{s3} by $1-\ve$ we see that $x\mapsto(1-\ve)\varphi(x)$ is an antihomomorphism.
\end{proof}

We need three more lemmas to obtain the main result of this section. The first one is known from the theory of functional identities, see  \cite[Example 1.5]{FIbook}.

\begin{lemma} \label{lbid} If $B$ is  a unital algebra such that $B=K(B)$, then every biderivation  of $B$ is inner.
\end{lemma}

The second lemma is a byproduct of Lemma \ref{lef1}.

\begin{lemma}\label{lefi2}
If $A$ is a unital algebra such that   $A=K(A)$,   then $A =    [[A,A],A]\circ A$.
\end{lemma}

\begin{proof}
    Let $T_1=([[x_1,x_2],x_3])^{\rm T}$. Clearly,
    $T_1$ is a nonmatrix commutator T-ideal. Hence by Lemma \ref{lef1}, $T_1(A)=A$, i.e.,
    $A[[A,A],A] A=A$. Since $[[A,A],A]$ is  a Lie ideal (i.e., $[[[A,A],A],A]\subseteq [[A,A],A]$),  we can write this as $A= [[A,A],A]A$. Using {\rm char}$(F)\ne 2$, it follows that
    $$A =  [[[A,A],A],A] +  [[A,A],A]\circ A.$$
   The first term lies in $[[A,A],A]$, and hence $A =    [[A,A],A]\circ A$.
\end{proof}

The third lemma is about Jordan ideals of associative algebras---recall that these are  linear subspaces $J$ of $A$ satisfying $J\circ A\subseteq J$.

\begin{lemma}\label{lm}
If $A$ is a unital algebra such that   $A=K(A)$,    then
   every Jordan ideal of $A$ is an ideal.
\end{lemma}

\begin{proof}
    By \cite[Corollary 2.2]{BFF}, it is enough to show that $A$ coincides with the subalgebra $\langle [A,A]\rangle$ generated by $[A,A]$. As noticed in the  proof of the preceding lemma, to this end it suffices to show that 
    $\langle [A,A]\rangle$ contains the ideal $A[[A,A],A]A = [[A,A],A]A$---and this is true as we see  from
$[[x,y],z]w = [[x,y]w,z]-[x,y][w,z]$.    
\end{proof}

\begin{theorem}\label{mt1}
 If $A$ is a unital algebra such that   $A=K(A)$, then  every Jordan epimorphism 
  $\varphi:A\to B$, where $B$ is any $F$-algebra, is standard. Moreover, there exists an idempotent $\ve \in Z(B)$ such that $x\mapsto \ve \varphi(x)$ is a homomorphism and $x\mapsto (1-\ve) \varphi(x)$ is an antihomomorphism.
\end{theorem}

    \begin{proof}  
Lemma \ref{lm} implies that $\ker\varphi$ is an ideal of $A$. Therefore,
    $\varphi$ induces a Jordan isomorphism from the algebra $A/\ker\varphi$ onto the algebra $B$.
Since $A/\ker\varphi$ is also unital and coincides with its commutator ideal, we may assume without loss of generality that $\varphi$ itself  is a Jordan isomorphism.

    Proposition \ref{pp} tells us that $\varphi$
    is nearly-standard. It is therefore enough to show that $B$ 
 satisfies
the assumptions of Lemma \ref{ts}.

First, note that $B$ is also  unital \cite[Corollary 3]{JR}.
Since $A =    [[A,A],A]\circ A$ by Lemma \ref{lefi2} and since \eqref{uava} shows that
$\varphi$ satisfies $$\varphi([[a_1,a_2],a_3]\circ a_4)=[[[\varphi(a_1),\varphi(a_2)],\varphi(a_3)]\circ \varphi(a_4),$$ we have
$B=[[B,B],B]\circ B$. In particular, $B=K(B)$, so   Lemma \ref{lbid} tells us that every biderivation of $B$ is inner. Finally, $B=[[B,B],B]\circ B$ also shows that the second assumption of Lemma \ref{ts} holds. This completes the proof of the theorem.
    \end{proof}

    We leave as an open question whether the assumption that $\varphi$ is surjective  is necessary. As the algebra $B$ is arbitrary, we can rephrase this assumption as  that the image of $\varphi$ is an associative algebra.

    Our aim in  the rest of the section is
to show that  the basic assumption of 
 Proposition \ref{pp} and Theorem \ref{mt1} that $A=K(A)$ is, from a certain viewpoint, close to necessary.
We start with a proposition providing conditions for the existence of non-standard Jordan homomorphisms. It was inspired by \cite[Example 1]{JR}.

\begin{proposition}\label{lapp}
 Let $A$ be a unital algebra, let
 $d$ and $g$ be derivations of $A$, let 
$J$ be an ideal of $A$ and   $\pi:A\to A/J$ be the natural projection.
Then  $B=A\times A/J$ becomes an associative algebra under  componentwise linear structure and multiplication given by
    $$(x,u)\star (y,v) = (xy, \pi(d(x)g(y))+ \pi(x)v+u\pi(y) ).$$
If  
\[[g(x),d(x)]\in J\quad\mbox{for all $x\in A$,}\]
then $\varphi:A\to B$,
$$\varphi(x)= \bigl(x,\frac{1}{2}\pi((dg)(x))\bigr)$$
is a Jordan homomorphism.
Moreover, if there 
exist  $s,t\in A$  such that 
$$g(s)d(t)-d(s)g(t)\notin J,$$
then $\varphi$ is not standard. 
\end{proposition}

\begin{proof}
Note first that $\star$ is an associative operation, so $B$ is indeed an algebra. It is straightforward to check that $\varphi$ is a Jordan homomorphism if $[g(x),d(x)]\in J$ for every $x\in A$.

Assume that $\varphi$ is standard, i.e., $\varphi=\varphi_1+\varphi_2$ where  $\varphi_1:A\to B$ is a homomorphism, $\varphi_2:A\to B$  is an antihomomorphism, and $\varphi_1(A)\star\varphi_2(A)= \varphi_2(A)\star\varphi_1(A)=(0)$. Write $$\varphi_1(x)=(\psi_1(x), \theta_1(x))\quad\mbox{and} \quad\varphi_2(x)=(\psi_2(x), \theta_2(x)),$$ where $\psi_i:A\to A$, $\theta_i:A\to A/J$, $i=1,2$. It is immediate that $\psi_1$ is a homomorphism of $A$, $\psi_2$ is an antihomomorphism of $A$, $\psi_1(A)\psi_2(A)=\psi_2(A)\psi_1(A)=(0)$, and $$x=\psi_1(x)+\psi_2(x), \quad \frac{1}{2}\pi((dg)(x)) =\theta_1(x)+\theta_2(x)$$ for every $x\in A$. 
In particular, $1=e+f$ where
$e=\psi_1(1)$ and $f=\psi_2(1)$. Obviously, $e$ and $f$ are orthogonal idempotents and $$\psi_1(x)=\psi_1(x)e = e\psi_1(x),\quad \psi_2(x)=\psi_2(x)f = f\psi_2(x).$$ From $(x-\psi_1(x))e =\psi_2(x)\psi_1(1) =0$ 
we see that $\psi_1(x)= xe$. Similarly, $\psi_1(x)=ex$, so $e$ is a central idempotent. Accordingly, $\psi_2(x)=fx$
and, since $\psi_2$ is an antihomomorphism, 
$fA$ is a commutative ideal of $A$. 

As derivations send central idempotents to $0$,   it follows from
$\varphi_1(1)\star\varphi_2(x)=0$
that $$\pi(e)\theta_2(x) + \theta_1(1)\pi(f)\pi(x)=0.$$
Multiplying by $\pi(e)$ we obtain $\pi(e)\theta_2(x)=0$. Therefore, $\theta_2(x)=\pi(f)\theta_2(x)$. Similarly, $\theta_1(x)=\pi(e)\theta_1(x)$. Consequently,
\begin{equation}
    \label{lisb}
\pi(e)\theta_1(x) + \pi(f)\theta_2(x)=\frac{1}{2}\pi((dg)(x)).\end{equation}

Now take any $s,t\in A$. Our goal is to show that 
$g(s)d(t)-d(s)g(t)\in J$. To this end,  consider $\varphi(st)$. On the one hand, 
\begin{align*}
    &\varphi(st)=\bigl(st, \frac{1}{2}\pi((dg)(st))\bigr) \\=& \bigl(st, \frac{1}{2}\pi((dg)(s))\pi(t) + \frac{1}{2}\pi(d(s)g(t)) +\frac{1}{2}\pi(g(s)d(t)) + \frac{1}{2}\pi(s)\pi((dg)(t))\bigr). 
\end{align*}
On the other hand,
\begin{align*}
\varphi(st)=&\varphi_1(st)+ \varphi_2(st)=\varphi_1(s)\star \varphi_1(t) + \varphi_2(t)\star \varphi_2(s)\\
=&\bigl(es,\theta_1(s))\star (et,\theta_1(t)) + (ft,\theta_2(t))\star (fs,\theta_2(s)\bigr).\end{align*}
Using the commutativity of $\pi(f)A/J$ and the fact that
 $d$ and $g$ vanish on central idempotents $e$ and $f$,  it  follows that
\begin{align*}
\varphi(st)=& \bigl(est, \pi(ed(s)g(t)) + \pi(e)\pi(s)\theta_1(t) + \pi(e)\theta_1(s)\pi(t)\bigr)\\& +
\bigl(fst, \pi(fg(s)d(t)) +\pi(f)\theta_2(s)\pi(t) + \pi(f)\pi(s)\theta_2(t)\bigr)\\
=& \bigl(st, \pi(ed(s)g(t)) + \pi(fg(s)d(t))  \\&+\pi(s)(\pi(e)\theta_1(t) +\pi(f)\theta_2(t)) + (\pi(e)\theta_1(s)+\pi(f)\theta_2(s))\pi(t)\bigr).
\end{align*}
Comparing both expressions and using \eqref{lisb}, we obtain
$$\frac{1}{2}d(s)g(t)  + \frac{1}{2} g(s)d(t) - ed(s)g(t) - fg(s)d(t)\in J.$$
Multiplying by $e$, it follows that $$e(g(s)d(t)-d(s)g(t))\in J,
$$ 
and multiplying by $f$, it follows that $$f(g(s)d(t)-d(s)g(t))\in J.
$$ 
As $e+f=1$,  the desired conclusion follows.  
\end{proof}

\begin{example} \label{snx}
Proposition \ref{lapp} shows that
all we need for the existence of a non-standard Jordan homomorphism from a commutative algebra $A$  to some algebra $B$ is a pair of derivations $d,g$ of $A$ that satisfy $g(s)d(t)\ne d(s)g(t)$ for some $s,t\in A$. Concrete examples can be easily found. As already mentioned above, such Jordan homomorphisms     are automatically nearly-standard.
\end{example}

Let us also give an example of a non-standard but nearly-standard Jordan homomorphism on a noncommutative algebra.

\begin{example}\label{examb}Let $A$ be the Grassmann algebra on $n$ generators $e_1,\dots,e_n$, where $n\ge 4$. For any $a,b,x\in A$, we have
$$[a,x][b,x] = [b,[ax,x]] -[b,[a,x]]x.$$ 
Since $[[A,A],A]=(0)$, this gives $[a,x][b,x] = 0$. Thus, every pair of inner derivations $d$ and $g$ satisfies $$d(x)g(x)=0=g(x)d(x)$$ for every $x\in A$. Taking $d(x)=[e_1,x]$ and $g(x)=[e_2,x]$, we have $$g(e_3)d(e_4) - d(e_3)g(e_4)=
8e_1e_2e_3e_4\ne 0.$$
Therefore, $d$ and $g$ satisfy conditions of Proposition \ref{lapp} (for $J=(0)$). Since $dg=0$, $\varphi(x)=(x,0)$ is a Jordan homomorphism that is not standard. However,
$\varphi$ is nearly-standard. Indeed, the maps
$$\varphi_1(x)= (x,e_1[e_2,x])\quad\mbox{and}\quad\varphi_2(x)= (0,-e_1[e_2,x])$$ satisfy $\varphi=\varphi_1 + \varphi_2$ and
$\varphi(xy)=\varphi(x)\star\varphi_1(y) + \varphi_2(y)\star\varphi(x)$ for all $x,y\in A$.
\end{example}

\begin{corollary}\label{cfin}
    Let $A$ be a unital algebra.
\begin{enumerate}
    \item[\rm (a)] If $A=K(A)$, then, for any commutative unital $F$-algebra $C$ and any $F$-algebra $B$, every Jordan epimorphism $\varphi:A\otimes C\to B$ is standard.
     \item[\rm (b)] If $A=K(A)$, then, for any commutative unital $F$-algebra $C$ and any $F$-algebra $B$, every Jordan homomorphism $\varphi:A\otimes C\to B$ is nearly-standard.
       \item[{\rm (c)}] If $A\ne K(A)$, then there exists a commutative unital algebra $C$, an $F$-algebra $B$, and a 
       Jordan homomorphism $\varphi:A\otimes C\to B$ that is not standard.
       
       \item[{\rm (d)}] If $A\ne K(A)+Z(A)$, then there exists a commutative unital algebra $C$, an $F$-algebra $B$, and a 
       Jordan homomorphism $\varphi:A\otimes C\to B$ that is not nearly-standard.
\end{enumerate}
    \end{corollary}
   
\begin{proof}
Observe that $K(A\otimes C) = K(A)\otimes C$ for any commutative unital algebra $C$. Therefore, $K(A)=A$ implies
$K(A\otimes C)=A\otimes C$.  Thus, (a) follows from Theorem \ref{mt1} and (b) follows from Proposition \ref{pp}.

To prove (c) and (d), take
$C$ to be the polynomial algebra $F[x_1,x_2]$. Note that $d= 1_A\otimes \frac{\partial \,}{\partial x_1}$ and $g= 1_A\otimes \frac{\partial \,}{\partial x_2}$ are derivations of $A\otimes C$. Let $J=K(A\otimes C)$, let $\pi:A\otimes C\to (A\otimes C)/J$ be the natural projection,
and let $$B=(A\otimes C)\times (A\otimes C)/J$$ be the algebra  from Proposition \ref{lapp}. The condition
$[d(x),g(x)]\in J$, $x\in A\otimes C$, is  fulfilled since $J$ is the commutator ideal. Therefore,
$$\varphi(x)=\bigl(x,\frac{1}{2}\pi((dg)(x))\bigr)$$ is a Jordan homomorphism from $A\otimes C$ to $B$.

Assume that $A\ne K(A)$. Then $J=K(A)\otimes C\ne A\otimes C$.
Write $s=1\otimes x_2$ and $t=1\otimes x_1$. Then $g(s)=d(t)=1\otimes 1$ and
$d(s)=g(t)=0$. Accordingly,
$g(s)d(t)-d(s)g(t)=1\otimes 1\notin J$ and hence  Proposition \ref{lapp} tells us that $\varphi$  is not standard. This proves (c).

Finally, assume that $A\ne K(A) + Z(A)$. Suppose that
$\varphi$ is nearly-standard. As noticed at the beginning of the proof of Lemma \ref{ts}, this implies that there exists a mapping $\varphi_2:A\otimes C\to B$ such that
\begin{equation}\label{lrside}
    \varphi(xy)- \varphi(x)\star\varphi(y)= \varphi_2(y)\star \varphi(x) - \varphi(x)\star \varphi_2(y)
\end{equation}
for all $x,y\in A\otimes C$. A short calculation shows that the left-hand side of \eqref{lrside} is equal to
$$\frac{1}{2}\bigl(0,\pi(g(x)d(y)-d(x)g(y))\bigr).$$
Write $\varphi_2(y)=(\alpha(y),\beta(y))$, where
$\alpha:A\otimes C\to A\otimes C$ and $\beta:A\otimes C\to (A\otimes C)/J$. Using the commutativity of $(A\otimes C)/J$, one directly computes that the right-hand side of \eqref{lrside} is equal to
$$\bigl([\alpha(y),x], \pi (g(x)d(\alpha(y)) - d(x)g(\alpha(y)))\bigr).$$
Comparing the above two expressions, we obtain
$$[\alpha(x),y]=0\quad\mbox{and}\quad \pi\bigl(g(x)d(y-2\alpha(y)) - d(x)g(y-2\alpha(y))\bigr)=0$$
for all $x,y\in A\otimes C$. The first equation shows that $\alpha(x)\in Z(A\otimes C)$ for every $x\in A\otimes C$. Write $1\otimes x_2$ for $x$ 
 in the second equation. Then
$g(x)=1\otimes 1$ and $d(x)=0$, 
so we obtain
$d(y - 2\alpha(y))\in J$ for every $y\in A\otimes C$. 
Since $\alpha(y)$ lies in $ Z(A\otimes C)$, so does $d(\alpha(y))$. Therefore,
$d(y)\in Z(A\otimes C) + J$. In particular,
$$a\otimes 1 =d(a\otimes x_1) \in Z(A\otimes C) + J = Z(A)\otimes C + K(A)\otimes C$$
for every $a\in A$.
However, this cannot hold for $a$ not lying in $K(A) + Z(A)$. With this contradiction, we  have completed the proof of  (d).
\end{proof}

Recall that every Jordan homomorphism on a commutative algebra is nearly-standard. The assumption that $A\ne K(A)+Z(A)$ in (d) therefore cannot be replaced by $A\ne K(A)$.

\section{Jordan homomorphisms on a certain T-ideal}\label{sec3}

As above, we assume throughout this section that $A$ and $B$  are  associative $F$-algebras, where $F$ is a field with char$(F)\ne 2$.

Let $A^{\rm op}$ denote the opposite algebra of $A$. That is, $A^{\rm op}$ has the same linear structure, but $a^{\rm op}b^{\rm op} = (ba)^{\rm op}$ for all  $a,b\in A$. Consider the natural Jordan homomorphism $s:A\to A\oplus A^{\rm op}$,
$s(a)= a + a^{\rm op}$. The image of $s$ is the Jordan algebra of symmetric elements with respect to the exchange involution ${\rm ex}:a+ b^{\rm op}\mapsto b + a^{\rm op}$.

It is easy to see that a Jordan homomorphism $\varphi:A\to B$ is  standard if  there exists a homomorphism $\chi:A\oplus A^{\rm  op}  \to B$  such that the diagram
$$\xymatrixcolsep{5pc}\xymatrixrowsep{3pc}\xymatrix{A  \ar@{->}[r]^s \ar@{>}[dr]^{\varphi} & A\oplus A^{\rm  op} \ar@{>}[d]^{\chi}
  \\
 & B
}
$$
is commutative. 


Let $J$ be either $A$ or $H(A,*)$
and $\varphi:J\to B$ be a Jordan homomorphism. Write $B'$
for $\langle \varphi(J)\rangle$, i.e., $B'$ 
is the subalgebra of $B$ generated by the image $\varphi(J)$ in $B$. Let 
$${\rm Ann}(B')=\{b\in B'\,|\, bB'=B'b=(0)\}.$$ 
Clearly, $\varphi:J\to B$ gives rise to the Jordan homomorphism $\overline{\varphi}:J\to B'/{\rm Ann}(B')$, the composition
of $\varphi$ and the natural projection $B'\to B'/{\rm Ann}(B')$. We call $\varphi$ {\bf annihilator-by-standard} if $\overline{\varphi}$ is standard.

Every Jordan homomorphism $\varphi:A\to B$ gives rise to the Jordan homomorphism $H(A\oplus A^{\rm op},{\rm ex})\to B$,
$a+a^{\rm op}\mapsto \varphi(a)$. The Jordan homomorphism $\varphi$ is standard (resp.\ annihilator-by-standard) if and only if the Jordan homomorphism $H(A\oplus A^{\rm op},{\rm ex})\to B$ is standard (resp.\ annihilator-by-standard).

Let $X=\{x_1,x_2,\dots\}$. Let $F\langle X\rangle$ be the free associative algebra on the set of generators $X$.
When considering T-ideals of $F\langle X\rangle$, we always assume that they are
closed with respect to partial linearizations.

Let $G\triangleleft F\langle X\rangle$ 
be a ${\rm T}$-ideal. If $J=A$ then 
$J_G=G(A)$, and if $J=H(A,*)$ then $J_G = H(G(A),*)$. It is clear that the ideal $G(A)$ is invariant with respect to any involution.

The following is the main result of this section.

\begin{theorem}\label{te}
    There exists a ${\rm T}$-ideal $G\triangleleft F\langle X\rangle$ with the following properties:
    \begin{enumerate}
        \item If $A,B$ are associative algebras,
        $J=A$ or $J=H(A,*)$, and
        $\varphi:J\to B$ is a
        Jordan homomorphism, then the  restriction of $\varphi$ to $J_G$ is annihilator-by-standard.  
        \item Let $T(M_2)$ be the {\rm T}-ideal of identities of the algebra of $2\times 2$ matrices. Then every element of  
        $T(M_2)$ is nilpotent modulo the 
        ${\rm T}$-ideal  $G$.
    \end{enumerate}
\end{theorem}

The proof will be given after a series of lemmas.

We will consider only the case  $J=H(A,*)$. 
The case $J=A$ is easier and, besides, $A^{(+)}\cong H(A\oplus A^{\rm op}, {\rm ex})$ (as usual, $A^{(+)}$ is the Jordan algebra which is equal to $A$ as a linear space and is equipped with the product
$x\circ y= xy+yx$).

Without loss of generality we will assume that the field $F$ is algebraically closed. Indeed, let $\overline{F}$ be the algebraic closure of $F$. The involution 
$*$ extends to the $\overline{F}$-involution on $\overline{A}=A\otimes _F \overline{F}$ and the Jordan homomorphism $\varphi$ extends to the Jordan homomorphism $\overline{\varphi}:H(\overline{A},*)\to
\overline{B}=B\otimes _F \overline{F}$. We have $G(\overline{A})=G(A)\otimes_F \overline{F}$ and 
Ann$(\overline{B})=
{\rm Ann}(B)\otimes_F \overline{F}$. If the homomorphism $\overline{\varphi}$
is annihilator-by-standard, then so is the homomorphism $\varphi$.

In the Jordan algebra $F\langle X\rangle^{(+)}$
we consider the Jordan subalgebra
$SJ\langle X\rangle$ generated by $X$. We refer to elements from $SJ\langle X\rangle$ as  Jordan polynomials.

Let $J\langle X\rangle$ be the free Jordan $F$-algebra on the set of free generators $X$.

\begin{lemma}\label{lema1}
    For an arbitrary ${\rm T}$-ideal 
    $P$ of the free Jordan algebra $J\langle X\rangle$ such that
    $P(SJ\langle X\rangle)\ne (0)$ there exists
    a nonzero ${\rm T}$-ideal $\widetilde{P}$ of the free associative algebra $F\langle X\rangle$
    such that $\widetilde{P}(A)\subseteq AP(H(A,*))A$.
\end{lemma}

\begin{proof}
    By the result of S.\ A.\ Amitsur \cite{Am1}, the factor algebra $ A/AP(H(A,*))A$
satisfies a polynomial identity that depends only on the ${\rm T}$-ideal $P$. This implies the assertion of the lemma.\end{proof}

 In what follows $\widetilde{P}$ denotes the largest ${\rm T}$-ideal in $F\langle X\rangle$ 
 satisfying the conclusion of Lemma \ref{lema1}.

It is well known
 that if $I$ is an ideal of a  Jordan algebra $J$, then its cube $I^3=(I\circ I)\circ I$ is also an ideal of $J$  (see \cite{cube}).

 Let $P\triangleleft J\langle X\rangle$ be a ${\rm T}$-ideal. Denote by $P'$ the ideal of $J\langle X\rangle$ generated by all elements 
 $(x\circ y)\circ z-(x\circ z)\circ y$, where $x\in P^3$,
 $y,z\in P$.
 If $x,y,z$ are elements of an associative algebra $A$ then
\begin{equation}\label{triple}
    (x\circ y)\circ z - (x\circ z)\circ y = [x,[y,z]].
\end{equation}
 Hence, for the Jordan algebra $J=H(A,*)$ we have
 $$P'(J)\subseteq A[P(J)^3,[P(J),P(J)]]A.$$

 The next lemma will be  repeatedly used in the case where $I=P(J)$, but we will also need it in another situation.
 \begin{lemma}\label{lema2}
   Let $A$ be an  associative algebra (with multiplication denoted by juxtaposition). Let  $J$ be an arbitrary subalgebra of the
   Jordan algebra $A^{(+)}$, let $I$ be an ideal of the Jordan algebra $J$, and let $I^3$ be the cube of the ideal $I$ in the Jordan algebra $J$. Then $JI^3\subseteq II$ and  $I^3J\subseteq II$.
 \end{lemma}

 \begin{proof}
     The space $I^3$ is spanned by elements $uvu$, $u,v\in I$. For an arbitrary element $a\in J$ we have
     $$auvu=(auv + vua)u - v(uau)\in II.$$
This proves the first inclusion. The proof of the second inclusion is similar. \end{proof}

Let $$A_P =\langle P(J)\rangle,$$ i.e., $A_P$ is the subalgebra 
of the associative algebra $A$ generated by $P(J)$.

\begin{lemma}\label{lema3}
    $AP'(J)A\subseteq A_P$.
\end{lemma}

\begin{proof} 
    Since $A$ is the direct sum
    of $H(A,*)$ and the space of skew-symmetric elements $K(A,*)=\{a\in A\,|\, a^*=-a\}$, it is enough to show that   $aP'(J)b\subseteq A_P$ for any 
 elements $a,b\in H(A,*)\cup K(A,*)$. Moreover, to this end it is enough to 
  show that   $$a[P(J)^3,[P(J),P(J)]]b\subseteq A_P$$
  for any 
  $a,b\in H(A,*)\cup K(A,*)$.
 
    Let us start by showing
   $P'(J)b\subseteq A_P$ for an arbitrary element $b\in H(A,*)\cup K(A,*)$.  If $b\in H(A,*)$ then the assertion follows from Lemma \ref{lema2}.
   Let $b\in K(A,*)$. Choose elements $p\in P(J)^3$, $p_1,p_2\in P(J)$. We have
 \begin{equation}
     \label{gegn}
[p_1,p_2]b =\frac{1}{2}([p_1,p_2]\circ b + [[p_1,p_2],b]) \end{equation}
The mapping $J\to J$,
$x\mapsto [x,b]$, is a derivation. Since $P$ is a T-ideal, it follows that $[P(J),b]\subseteq P(J)$. Therefore, \eqref{gegn} implies 
that
$$[p_1,p_2]b \in H(A,*) + [P(J),P(J)].$$
Similarly, $$b[p_1,p_2]\in H(A,*) + [P(J),P(J)].$$
   Hence, 
   $$p[p_1,p_2]b\in P(J)^3 H(A,*)+A_P =A_P$$
   by Lemma \ref{lema2}.
Similarly, since $[p,b]\in P(J)$ implies $[p_1,p_2][p,b]\in A_P$, it follows that
$$[p_1,p_2]pb = [p_1,p_2][p,b] + [p_1,p_2]bp\in A_P.$$

Now choose arbitrary elements $a,b\in H(A,*)\cup K(A,*)$. If $a\in H(A,*)$ then
$$aP'(J)b\subseteq (a\circ P'(J))b+P'(J)ab\subseteq P'(J)b + P'(J)ab\subseteq A_P$$
   by the above. Let 
$a\in K(A,*)$. Then the mapping
$J\to J$, $x\mapsto [a,x]$ is a derivation of the Jordan algebra $J$. Hence,
$$aP'(J)b\subseteq [a,P'(J)]b + P'(J)ab\subseteq P'(J)b + P'(J)ab\subseteq A_P.
$$
This completes the proof of the lemma.
\end{proof}

Let $U$, $u:J\to U$, be  the special enveloping algebra of the Jordan algebra $J$. This means that
there exists a Jordan homomorphism 
$u:J\to U$ such that the associative algebra $U$ is generated by the image $u(J)$ and for any Jordan homomorphism $\varphi:J\to B$ there exists a homomorphism $\chi: U\to B$ of associative algebras such that the diagram
$$\xymatrixcolsep{5pc}\xymatrixrowsep{3pc}\xymatrix{J  \ar@{->}[r]^u \ar@{>}[dr]^{\varphi} & U \ar@{>}[d]^{\chi}
  \\
 & B
}
$$
is commutative. The algebra $U$ is equipped with an involution $\tau:U\to U$ leaving all elements $u(a)$, $a\in J$, fixed. For more information about universal special algebras, see \cite{Jac,  McC, fourauthors}.

Applying the above definition to the Jordan embedding $H(A,*)\to A$,
$a\mapsto a$, we obtain a homomorphism $t:U\to A$, $t(u(a))=a$, $a\in H(A,*)$.  Consider the ideal $I=\ker t$. Clearly, $I^\tau = I$.
Denote $$I_P=I\cap U_P.$$

In \cite{Z}, the second author constructed the ${\rm T}$-ideal $T$ of the free Jordan algebra $J\langle X\rangle$ with the following properties:
\begin{enumerate}
    \item If $a_1\in T(SJ\langle X\rangle)$
    and $a_2,a_3,a_4 \in SJ\langle X\rangle$, then 
    $$a_1a_2a_3a_4 + a_4a_3a_2a_1\in SJ\langle X\rangle$$
    (we call this the {\bf tetrad-eating property} of $T$).
    \item If $J$ is a finite-dimensional simple Jordan algebra that is not an algebra of a symmetric bilinear form, then $T(J)=J$.
\end{enumerate}

From (1) and the result by P.\ M.\ Cohn 
\cite{Cohn}  it follows that for arbitrary elements 
$a_1,\dots,a_n \in T(SJ\langle X\rangle)$, $n\ge 1$, we have
\begin{equation*}
    \label{p3} a_1 \cdots a_n + a_n\cdots a_1 \in T(SJ\langle X\rangle).
\end{equation*}
This implies that
$$H(U_T,\tau)= u(T(J)).$$
Hence
$$I\cap H(U_T,\tau)=(0).$$
Clearly, $I_T^\tau = I^\tau \cap U_T^\tau =I_T$ and therefore
\begin{equation*}\label{by1}
    H(I_T,\tau)=(0);
\end{equation*}
in other words, every element from $I_T$ is skew-symmetric with respect to $\tau$. By Lemma \ref{lema2},
$JT^3(J)\subseteq  T(J)T(J)$,
hence $u(J)U_{T^3}\subseteq U_T$,
and therefore
\begin{equation}\label{by2}u(J)I_{T^3}\subseteq I_T.\end{equation}

\begin{lemma}\label{lema4}
    $I_{T^3}$ lies in the center of the algebra $U$.
\end{lemma}

\begin{proof}
    Choose elements $x\in u(J)$,
    $y\in I_{T^3}$. We have     $y^\tau=-y$. By
  \eqref{by2}, we also have
    $(xy)^\tau=-xy$. Hence,
    $$(xy)^\tau = y^\tau x^\tau =-yx=-xy,$$
    which implies the assertion of the lemma.
\end{proof}

Recall that $T'$ is the Jordan ${\rm T}$-ideal generated by $(x\circ y)\circ z - (x\circ z)\circ y$, where $x\in T^3$, $y,z\in T$. Clearly, $T'\subseteq T^3$.

\begin{lemma}\label{lema5}
    $I_{T'}U_{T'} = U_{T'}I_{T'} =(0)$.
\end{lemma}

\begin{proof}
From Lemma \ref{lema4}    it follows that $$I_{T^3}[u(T^3(I)),U]=(0).$$ 
In terms of the associative multiplication in the algebra 
$F\langle X\rangle$, \eqref{triple} holds.
Hence,
$$T'(SJ\langle X\rangle)\subseteq F\langle X\rangle[T^3(SJ\langle X\rangle),F\langle X\rangle]F\langle X\rangle.$$
This implies the assertion of the lemma.
\end{proof}

Now we are ready to introduce the ${\rm T}$-ideal $G$ and to prove the main theorem.

Let $T''=(T')'\triangleleft J\langle X\rangle$ and
let $$G=\widetilde{T''}$$ be the ${\rm T}$-ideal of $F\langle X\rangle$.

\bigskip
\noindent{\em Proof of Theorem \ref{te}.} Our goal is to show that the restriction of $\varphi$ to
$H(G(A),*)$ is annihilator-by-standard. Let $B'$ be the subalgebra of $B$ generated by $\varphi(H(G(A),*))$. We need to show that the Jordan homomorphism
$$H(G(A),*)\to B'/{\rm Ann}(B')$$
is standard.

Let $B_T'$ be the subalgebra of $B$ generated by $\varphi(T'(J))$
or, equivalently, $B_{T'}= \chi(U_{T'})$. By Lemma \ref{lema5},
$\chi(I_{T'})\subseteq {\rm Ann}(B_{T'})$.

As above, we denote the subalgebra of $A$ generated by $T'(H(A,*))$ by
$A_{T'}$. 
Let us show that 
 the mapping
$$T'(H(A,*))\to  B_{T'}/{\rm Ann}(B_{T'}),$$
which is the composition of the mapping $u$ and $U_{T'}\to B_{T'}/{\rm Ann}(B_{T'})$, extends to a homomorphism of associative algebras 
$$\psi:A_{T'}\to B_{T'}/{\rm Ann}(B_{T'}).$$
By the universality of the mapping $u$ there is a homomorphism $\mu_1:U_{T'}\to B_{T'}$ and there is the mapping $\mu_2:U_{T'}\to A_{T'}$. The kernel
$\ker \mu_2=I_{T'}$ is mapped by $\mu_1$ to the annihilator Ann$(B_{T'})$. Hence there is a homomorphism $A_{T'}\to B_{T'}/{\rm Ann}(B_{T'})$ that extends the Jordan homomorphism $T'(H(A,*))\to  B_{T'}/{\rm Ann}(B_{T'})$.

By Lemmas \ref{lema1} and \ref{lema3},
$$G(A)\subseteq AT''(J)A\subseteq A_{T'}.$$ 
By the 
tetrad-eating property of the ideal $T$,
$$H(G(A),*)\subseteq H(A_{T'},*)\subseteq T'(H(A,*)).$$
Let $\psi_G$ be the restriction of the homomorphism $\psi$ to $G(A)$. Let $B_G$ be the subalgebra of $B$ generated by the image of $H(G(A),*)$. We have
$$B_G\cap {\rm Ann}(B_{T'})\subseteq {\rm Ann}(B_G),$$
$$\psi_G:G(A)\to B_G/B_G\cap {\rm Ann}(B_{T'}).$$
The composition 
$$G(A)\to B_G/B_G\cap {\rm Ann}(B_{T'})\to B_G/{\rm Ann}(B_G)$$
extends the Jordan homomorphism $H(G(A),*)\to B_G/{\rm Ann}(B_G)$.
We proved that the Jordan homomorphism
$$\varphi:H(G(A),*)\to B_{T'}/{\rm Ann}(B_{T'})$$
is standard. This completes the proof of the first part of the theorem.

Recall that the complexity of a ${\rm T}$-ideal $G\triangleleft F\langle X\rangle$ is a maximal $n\ge 1$ such that the matrix algebra $M_n(F)$ satisfies all identities from $G$. Our aim is to prove that the complexity of the ${\rm T}$-ideal $G$  is equal to $2$.

For that we need to recall S.\ A. Amitsur's proof in \cite{Am1}. Recall  that the field $F$
is assumed to be algebraically closed. Let $P$ be a Jordan ${\rm T}$-ideal such that
$P(SJ\langle X\rangle) \ne (0)$.

Let $(A_i,*_i)$, $1\le i\le m$, be the list of all finite dimensional simple involutive associative algebras such that $P(H(A_i,*_i)) = (0)$, $1\le i\le m$. If $f=0$ is an identity in all algebras $A_i$, $1\le i\le m$, then some power $f^k$ lies in $\widetilde{P}$. 

If $P=T''$ then the list is short:
$A_1=F$, $A_2 =F\oplus F$ with the exchange involution, $A_3= M_2(F)$ with the orthogonal or symplectic involution, $A_4=M_2(F)\oplus M_2(F)^{\rm op}$
with the exchange involution.
Indeed, unless the Jordan algebra $H(A_i,*_i)$
is associative as in the cases of $A_1$ and $A_2$, $T''(H(A_i,*_i))=(0)$
if and only if $T(H(A_i,*_i))=(0)$.
In \cite{Z} it was shown that the latter equality holds only in the cases $A_3$ and $A_4$.

If $f\in T(M_2)$ then some power 
$f^k$ lies in the ${\rm T}$-ideal $G$. This completes the proof of the theorem.
 $\hfill\qed$

\bigskip

We will now construct an example of an associative algebra $A$, an associative algebra $B$ with an involution $*:B\to B$ and a Jordan isomorphism 
$\varphi:A\to H(B,*)$ such that for an arbitrary T-ideal ${\rm T}$, the restriction of $\varphi$ to $T(A)$ is not standard (Proposition \ref{pe2}).

We start with some well known facts concerning presentations of semigroups by generators and relations.

Let $X=\{x_1,x_2,\dots\}$ be a finite alphabet. Let $X^*$ be the set of all nonempty words in $X$. The semigroup $X^*$ is referred to as the free semigroup.

An equivalence relation $\sim$ on $X^*\times X^*$
is called a {\bf congruence} if $u_1\sim v_1$, $u_2\sim v_2$ implies $u_1u_2\sim v_1v_2$. This makes the set of equivalence classes a semigroup
$X^*/\sim$.

For a subset $\sigma\subset X^*\times X^*$, let $\langle \sigma\rangle$ be the minimal congruence containing $\sigma$. Two words $u,v$ are equivalent if there exists a sequence of words $u=u_1,u_2,\dots,u_n=v$ such that 
each pair $u_i,u_{i+1}$, $i=1,\dots,n-1,$
is of the form $u_i=u'u_i'v'$, $u_{i+1}=u'u_{i+1}'v'$, $u_i'\times u_{i+1}'$ or $u_{i+1}'\times u_i'$ lies in $\sigma$. In this case we we say that $u=v$ follows from $\sigma$.
The semigroup $S=X^*/\langle \sigma\rangle$ is {\bf presented} by the set of generators $X$ and the set of relations $\sigma$, $S=\langle X\,|\,\sigma\rangle$.

The semigroup is called {\bf finitely presented} if it has a presentation $S=\langle X\,|\,\sigma\rangle$ with both $X$
and $\sigma$ finite.

A semigroup $S$ is {\bf graded}  if no two words (in $X$) of different lengths are equal. Similarly we can talk about {\bf graded congruences}.

Given a graded congruence $\sim\,\subset X^*\times X^*$, we will define a new graded congruence
$\sim_{in}\,\subset X^*\times X^*$. Let $u,v\in X^*$, ${\rm length}(u)={\rm length}(v)=n$. We say that $u\sim_{in}v$ if $u=v$ follows from $\{u'\times v'\,|\, {\rm length}(u')={\rm length}(v') < n, u'\sim v'\}$. Clearly,
$$\sim_{in}\,\subset\, \sim.$$

Let $S$ be a semigroup. A mapping $f:S\times S\to F$ is called a {\bf 2-cocycle} if $f(ab,c)=f(a,bc)$ for any elements $a,b,c\in S$. 

Let $S=X^*/\sim$, where $\sim$ is a graded congruence. Let $h:S\times S\to F$ be a 
$2$-cocycle. Let us lift $h$ to $\tilde{h}:X^*\times X^*\to F$.

\begin{lemma}
    \label{lef2}  Let $a,b,c,d\in X^*$. If $ab\sim_{in} cd$ then $\tilde{h}(a,b)=\tilde{h}(c,d)$.

\end{lemma}

\begin{proof}
     We only need to consider the case where $ab=w'uw''$, $cd=w'vw''$, $u\sim v$, the words $w',w''$ are not simultaneously empty. Suppose that the word $w'$ is not empty, $w'=x_{i_1}w_i'$. Then the word $a$ starts with $x_{i_1}$, $a=x_{i_1}a'$. Hence,
    $\tilde{h}(a,b)=\tilde{h}(x_{i_1}, a'b)$, $a'b=w_1'uw''$.   Similarly,
    $\tilde{h}(c,d)=\tilde{h}(x_{i_1},w_1'vw'')$, the elements $w_1'uw''$ and $w_1'vw''$ are equal in $S$. This proves the lemma.\end{proof}

\begin{lemma}
   Let $\alpha:X^*\to F$ be a mapping such that $\alpha(u)=\alpha(v)$ if
        $u\sim_{in}v$. Define the $2$-cocycle $h:(X^*/\sim)\times (X^*/\sim) \to F$, $h(u,v)=\alpha(uv)$. The mapping $h$ is well defined.
\end{lemma}

\begin{proof}
We need to show that $u\sim u_1$, $v\sim v_1$ implies $\alpha(uv)=\alpha(u_1v_1)$, but $uv\sim_{in} u_1v_1$. This proves the lemma.
\end{proof}

Now suppose that the field $F$ is infinite, the mapping $\alpha:X^*/\sim_{in}\to F$ is injective. In other words, given words $u,v\in X^*$, we have $\alpha(u)=\alpha(v)$ if and only if
$u\sim_{in} v$.

Suppose that the graded semigroup 
$S=X^*/\sim$ is not finitely presented (there are many such semigroups).

\begin{lemma}
    \label{leff1} For an arbitrary $n\ge 1$ there exist words $u,v\in X^*$ such that 
    $u\sim v$, $u\not\sim_{in} v$,
    $
    {\rm length}(u)={\rm length}(v) \ge n$.
\end{lemma}

\begin{proof}
    It is sufficient to notice that since $S$ is not finitely presented there exist words $u,v$ such that $u=v$ in $S$,  $
    {\rm length}(u)={\rm length}(v) \ge n$, and $u=v$ does not follow from all equalities $u'=v'$ in $S$,  $
    {\rm length}(u')={\rm length}(v') < n$. This completes the proof of the lemma.
\end{proof}

Let $A$ be an associative algebra. Let $f:A\times A\to F$
be a bilinear form. 
Consider the algebra $\hat A= A \oplus Fz$ with multiplication $\hat A z =z \hat A = (0)$, 
$$a\cdot b = ab + f(a,b)z,$$ where $a,b\in A$, $ab$ is the product in the algebra $A$, and $a\cdot b$ is the product in the algebra $\hat A$. The algebra $\hat A$ is associative if and only if 
$f(ab,c)=f(a,bc)$ for arbitrary elements $a,b,c\in A$, i.e., $f$ is a (Hochschild) {\bf $2$-cocycle}. A $2$-cocycle is called a {\bf coboundary}  if there exists a linear functional $\lambda:A\to F$ such that $f(a,b)=\lambda(ab)$.

\begin{proposition}
      \label{pe1} There exists an algebra $A$
      and a $2$-cocycle $f:A\times A\to F$ with the following properties: for an arbitrary {\rm T}-ideal $P$, the ideal $P(A)$ contains elements $a_1,a_2,b_1,b_2$ such that $a_1b_1-a_2b_2=b_1a_1-b_2a_2=0$, $f(b_1,a_1)=f(b_2,a_2)=0$, but $f(a_1,b_1)\ne f(a_2,b_2)$. In particular, the restriction
      $f:P(A)\times P(A)\to F$ is not a coboundary.
\end{proposition}

\begin{proof}
Consider the semigroup algebra $FS$ and the algebra $A$ of strictly upper triangular infinite matrices over $FS$,
$$A=\sum_{1\le i <j} e_{ij}(FS),\quad e_{ij}(a)e_{pq}(b)=\delta_{jp} e_{iq}(ab). $$
We define a $2$-cocycle
\[ f(e_{ij}(u), e_{pq}(v))= \begin{cases} 
      h(u,v) & {\rm if}\,\,\,\,  i < j, p < q, j=p; u,v\in S \\
      0 & {\rm otherwise} 
   \end{cases}
\]

Let $P$ be a T-ideal. There exists a nonzero multilinear element
$$x_1\dots x_m + \sum_{1\ne \sigma \in S_m} k_\sigma x_{\sigma(1)}\dots x_{\sigma(m)}\in P,\,\,\,\, k_\sigma\in F.$$
This implies that 
$$e_{ij}(S^m)\subseteq P(A)$$
as soon as $j-i\ge m$. Choose elements $u,v\in X^*$,  $
    {\rm length}(u)={\rm length}(v) \ge 2m$,
$u\sim v$, but $u\not\sim_{in} v$.
Let $u=u'u''$, ${\rm length}(u),{\rm length}(u'') \ge m$, $v=v'v''$, ${\rm length}(v),{\rm length}(v'') \ge m$. The elements
$$a_1 = e_{1,1+m}(u'),\quad
b_1 = e_{1+m,1+2m}(u''),\quad
a_2 = e_{1,1+m}(v'),\quad
b_2 = e_{1+m,1+2m}(v''),$$ lie in $P(A)$, $a_1b_1-a_2b_2 = b_1a_1-b_2a_2=0$ in $A$, $f(b_1,a_1)=f(b_2,a_2)=0$, but 
$$f(a_1,b_1)=h(u',u'')\ne h(v',v'')=f(a_2,b_2).$$
This completes the proof.     
\end{proof}

Let $A$, $f:A\times A\to F$ be the algebra and the $2$-cocycle of Proposition \ref{pe1}. Consider the direct sum $\tilde{A}=A\oplus A^{\rm op}$ and extend the $2$-cocycle on $\tilde{A}:$
$$f(A,A^{\rm op})=f(A^{\rm op},A)=(0),\quad f(x^{\rm op}, y^{\rm op})=-f(y,x)$$
for arbitrary elements $x,y \in A$.
It is straightforward that this extension is a $2$-cocycle on the algebra $\tilde{A}$.

Consider the annihilator extension
$$B=A\oplus A^{\rm op} \oplus Fz.$$

\begin{lemma}
  {\rm (1)} The mapping $\ast:B\to B$, $a+b^{\rm op}+kz\mapsto b + a^{\rm op} -kz$, $k\in F$, is an involution.
  
      {\rm (2)} The mapping $\varphi:A\to H(B,\ast)$, 
        $a\mapsto a + a^{\rm op}$,
        is a Jordan isomorphism.
   
\end{lemma}

\begin{proof}
     We need to check that 
    $$\bigl((a+b^{\rm op})\cdot (c+d^{\rm op})\bigr)^*=(c+d^{\rm op})^*\cdot (a+b^{\rm op})^*.
    $$
    The left-hand side is
 \begin{align*}
 \bigl(ac + (db)^{\rm op} + (f(a,c)-f(d,b))z\bigr)^*  
 =db + (ac)^{\rm op} -(f(a,c)-f(d,b))z, 
 \end{align*}
 and the right-hand side is $$ (d+ c^{\rm op})\cdot (b+ a^{\rm op}) = 
 db+ (ac)^{\rm op} + (f(d,b)-f(a,c))z.$$
 This completes the proof of (1). Let us prove (2). We have 
 $$(a+a^{\rm op})^2 = a^2+(a^2)^{\rm op} + (f(a,a)+f(a^{\rm op}, a^{\rm op}))z=a^2 + (a^2)^{\rm op}.$$
 This completes the proof.
\end{proof}

\begin{proposition}\label{pez}
      \label{pe2} For  an arbitrary {\rm T}-ideal $P$, the restriction of $\varphi$ to $P(A)$ is not standard.
\end{proposition}

\begin{proof}
    By Proposition \ref{pe1}, $P(A)$ contains elements 
    $a_1,b_1,a_2,b_2$ such that
    $a_1b_1-a_2b_2=b_1a_1-b_2a_2=0$, but $f(a_1,b_1)-f(a_2,b_2)\ne 0$.

    Suppose that $\varphi=\varphi_1 + \varphi_2$, where $\varphi_1:P(A)\to B$ is a homomorphism, $\varphi_2:P(A)\to B$ is an antihomomorphism, and 
    $$\varphi_1(P(A))\varphi_2(P(A))
    = \varphi_2(P(A))\varphi_1(P(A))=(0).$$ Then 
    $$\varphi_1(a_1)\varphi_1(b_1)-\varphi_1(a_2)\varphi_1(b_2) = \varphi_1(a_1b_1-a_2b_2)=0$$
    and 
   $$\varphi_2(a_1)\varphi_2(b_1)-\varphi_2(a_2)\varphi_2(b_2) = \varphi_2(b_1a_1-b_2a_2)=0,$$
   hence $$\varphi(a_1)\varphi(b_1)-\varphi(a_2)\varphi(b_2)=0.$$
   However,
  \begin{align*}
&\varphi(a_1)\varphi(b_1)-\varphi(a_2)\varphi(b_2) \\=& (a_1 + a_1^{\rm op})\cdot (b_1 + b_1^{\rm op}) - (a_2 + a_2^{\rm op})\cdot (b_2 + b_2^{\rm op}) \\
      =& a_1b_1 + (b_1a_1)^{\rm op}
      + (f(a_1,b_1)-f(b_1,a_1))z\\&- a_2b_2 - (b_2a_2)^{\rm op} -
      (f(a_2,b_2)-f(b_2,a_2))z\\
      =& (f(a_1,b_1)-f(a_2,b_2))z\ne 0.
  \end{align*}
  This completes the proof.
\end{proof}

\end{document}